\documentclass[10pt, a4paper]{amsart}
\usepackage{amsmath, latexsym, amsfonts, amssymb, amsthm, amscd,comment}
\usepackage[pdftex]{graphicx}

\newtheorem{As}{Assumption}

\newtheorem{theorem}{Theorem}[section]

\newtheorem{cor}[theorem]{Corollary}
\newtheorem{prop}[theorem]{Proposition}

\newtheorem{lemma}[theorem]{Lemma}
\newtheorem{rem}[theorem]{Remark}
\newtheorem{defi}[theorem]{Definition}
\newcommand{\ii}{\mathbf i}
\newcommand{\ud}{\mathrm{d}}
\newcommand{\eq}{\mathbf{e}_q}
\newcommand{\eb}{\mathbf{e}_\beta}

\newcommand{\p}{\mathbb{P}}

\newcommand{\e}{\mathbb{E}}

\newcommand{\ed}{\stackrel{(d)}{=}}

\def\te#1{\mathrm{e}^{#1}}
\def\BEN{\begin{enumerate}}  \def\BI{\begin{itemize}}
\def\EEN{\end{enumerate}}   \def\EI{\end{itemize}}
    \def\nn{\nonumber}

\def\mbb{\mathbb}  
\def\mc{\mathcal} \def\unl{\underline} \def\ovl{\overline}

\def\le{\left}
\def\ri{\right}
\def\te#1{\mathrm{e}^{#1}}   
\def\WT{\widetilde}
\def\WH{\widehat} 
  
\def\I{\mathbf 1}

\def\g{\gamma}     \def\th{\theta}

  \def\nn{\nonumber}   
     
 \def\q{\qquad} 
   
  \def\td{\text{\rm d}}
 
\numberwithin{equation}{section}
\def\mbb{\mathbb}

\begin{document}
\title{On Future Drawdowns of
L\'evy processes}
\author{
E. J. Baurdoux}%
\address{
Department of Statistics, London School of Economics, Houghton Street, London WC2A 2AE, UK}
\email{e.j.baurdoux@lse.ac.uk}
\author{Z.
Palmowski}
\address{Faculty of Pure and Applied Mathematics,
Wroc\l{}aw University of Science and Technology,
Wyb. Wyspia\'nskiego 27, 50-370 Wroc\l{}aw, Poland}
\email{zbigniew.palmowski@gmail.com}
\author{M.R.
Pistorius}
\address{Department of Mathematics, Imperial College
London, South Kensington Campus, London SW7 2AZ, UK}
\email{m.pistorius@imperial.ac.uk}

\begin{abstract}\noindent
For a given L\'{e}vy process $X=(X_t)_{t\in\mbb R_+}$
and for fixed $s\in \mathbb{R}_{+}\cup\{\infty\}$ and $t\in\mathbb R_+$
we analyse the {\em future drawdown extremes} that are defined as follows:
\begin{eqnarray*}
\overline D^*_{t,s} = \sup_{0\leq u\leq t} \inf_{u\leq w < t+s}(X_w-X_u), \qquad\qquad
\underline D^*_{t,s} = \inf_{0\leq u\leq t} \inf_{u\leq w < t+s}(X_w-X_u).
\end{eqnarray*}
The path-functionals $\overline D^*_{t,s}$ and $\underline D^*_{t,s}$
are of interest in various areas of application, including financial mathematics and queueing theory.
In the case that $X$ has a strictly positive
mean, we find the exact asymptotic decay as $x\to\infty$ of
the tail probabilities $\p(\overline D^*_{t}<x)$ and $\p(\underline D^*_t<x)$ of
$\overline D^*_{t}=\lim_{s\to\infty}\overline D^*_{t,s}$
and $\underline D^*_{t} = \lim_{s\to\infty}\underline D^*_{t,s}$
both when the jumps satisfy the Cram\'er assumption and in a heavy-tailed case.
Furthermore, in the case that the jumps of the L\'{e}vy process $X$ are of single sign and $X$ is not subordinator, we
identify the one-dimensional distributions in terms of the scale function of $X$.
By way of example, we derive explicit results for the
Black-Scholes-Samuelson model.
\end{abstract}


\maketitle

\section{Introduction}
In recent times various pricing models with jumps have been put forward to address the shortcomings of diffusion models in representing the risk related to large market
movements (see e.g.~\cite{CT}).
Such models allow for a more realistic representation of price dynamics and a
greater flexibility in modeling and calibration of the model to market prices and
in reproducing a wide variety of implied volatility skews and smiles. An important indicator for the riskiness and effectiveness of an investment strategy is the drawdown, which is the distance of the current value away from the maximum value it has attained to date.
Various commonly used trading rules are based on the drawdown (see e.g. \cite{PosVec}), while drawdowns have also been deployed as risk-measure (see \cite{Had2,Had1}) and in the context of portfolio
optimisation (see \cite{CO,KOP}). Drawdown processes (also called reflected processes) are also encountered in various other areas, such as applied probability, mathematical genetics and queueing theory (see \cite{Debicki,krzysiekbook}).
See \cite{LLZ,MP,Z} and references therein for further applications and results concerning
drawdown processes.

In this paper we analyse a number of  path-functionals of the increments of a given
general L\'evy process $X=(X_t)_{t\in\mbb R_+}$ that are closely related to the drawdowns and drawups.
In particular, we consider the {\em future drawdown} and {\em future drawup extremes} that are defined by
for given $s,t\in\mathbb R_+$ by
\begin{align}
\overline D^*_{t,s} = \sup_{0\leq u\leq t} \inf_{u\leq w < t+s}(X_w-X_u), &\qquad\qquad
\underline D^*_{t,s} = \inf_{0\leq u\leq t} \inf_{u\leq w < t+s}(X_w-X_u),\label{St}\\
\overline U^*_{t,s} = \sup_{0\leq u\leq t} \sup_{u\leq w < t+s}(X_w-X_u), &\qquad\qquad
\underline U^*_{t,s} = \inf_{0\leq u\leq t} \sup_{u\leq w < t+s}(X_w-X_u),\label{Stu}
\end{align}
and we denote the infinite-horizon versions by
$$
\overline D^*_t = \lim_{s\to\infty}\overline{D}^*_{t,s},\q
\underline D^*_t = \lim_{s\to\infty}\underline{D}^*_{t,s},\q
\overline U^*_t = \lim_{s\to\infty}\overline{U}^*_{t,s},\q
\underline U^*_t = \lim_{s\to\infty}\underline{U}^*_{t,s}.
$$
The functionals $\overline D^*_{t,s}$, $\underline D^*_{t,s}$, $\overline U^*_{t,s}$ and
$\underline U^*_{t,s}$ are concerned with the variation in $u\in[0,t]$ of the smallest
and largest of the increments $\{X_w-X_u, w\in[u,t+s]\}$.  These functionals may be explicitly
represented in terms of the (maximal) drawdown and drawup (see Proposition~\ref{repr}).

Since, as is straightforward to check, we have $\underline D^*_{t,s} = - \overline{\widehat{U}}^*_{t,s}$ and
$\overline D^*_{t,s} = - \underline{\widehat{U}}^*_{t,s}$, where $\widehat{\cdot}$ denotes the quantity
calculated for the dual process $\widehat{X}=-X$, we may (and often do) restrict ourselves in subsequent analysis
to future drawdown extremes, without loss of generality.

The future drawdown and drawup processes arise in various applications, including in
financial risk analysis and queueing models.
We note that, under an exponential L\'evy model
$P_t=P_0\exp(X_t)$ for the stock price,
the random variables $\overline D^*_{t,s}$ and $\underline D^*_{t,s}$
are path-dependent risk indicators:
$\overline D^*_{t,s}$ and $\underline D^*_{t,s}$ are the maximal and
minimal values of the {\it lowest} future log-return
$\log(P_w/P_u)$ achieved for $w$ in the time-window $[u,t+s]$, where
$u$ is ranging over $[0,t]$. Another application comes from telecommunications and queueing models, where
$\overline U^*_t=\lim_{s\to\infty}\overline U^*_{t,s}$ and $\underline U^*_t=\lim_{s\to\infty}\underline U^*_{t,s}$
describe the supremum and the infimum of the workload process over a finite time horizon $t$ in a fluid model with netput $X$, respectively (see \cite{krzysiekbook} for a survey about L\'{e}vy-driven queues).

In the mentioned applications it is of interest to obtain the laws
of the random variables $\overline D^*_{t,s}$, $\underline D^*_{t,s}$, $\overline U^*_{t,s}$ and $\underline U^*_{t,s}$
for finite and infinite horizons $s$,
and in particular the tail-probabilities and their asymptotic behaviour.
Restricting ourselves to the case $s=\infty$ we identify the {\em exact} asymptotic decay as $x\to\infty$ of
the tail probabilities $\p(\overline D^*_{t,s}<-x)$ and $\p(\underline D^*_{t,s}<-x)$ of
$\overline D^*_{t,s}$ and $\underline D^*_{t,s}$.  We do so in the distinct cases of a light-tailed and a heavy-tailed L\'{e}vy measure. In the former setting we also consider the asymptotics when $x$ and $s$
tend to infinity in a fixed proportion.
Furthermore, when the jumps of $X$ are of single sign only and $X$ is not subordinator, we
explicitly identify the Laplace transform in time of the one-dimensional distributions
in terms of the scale function. As example, we analyze in detail (future) drawdowns and drawups under the
Black-Scholes model, identifying in particular
the mean of the value $P_t=P_0\exp(X_t)$ under the measure
$\underline{\p}^{(\gamma)}$ defined in (\ref{miara}) (for $\gamma$ given in Assumption \ref{A2})
and the laws of $\overline{D}^*_t$ and $\underline{D}^*_t$.

{\bf Contents.} The remainder of the paper is organized as follows.
In Section \ref{main} we present the main representation in terms of drawup and drawdown processes.
In Section \ref{sec: ascramer} we identify the Cram\'er asymptotics and describe the associated
drawup and drawdown measures in Section \ref{sec: asmeasures}. We analyse the heavy-tailed case in \ref{sec:heavy}.
 Finally, in  Section~\ref{Examplesexact} we derive exact distributions of future drawup and drawdowns in case $X$ has jumps of single sign and we present an application to the Black-Scholes model in Section~\ref{Examples}.

\section{Main representation}\label{main}
Let $(X_t)_{t\in\mathbb{R}_+}$ be a general L\'evy process ({\em i.e.}, a process
with stationary and independent increments with c\'{a}dl\'{a}g paths such
that $X_0=0$) defined on some filtered probability space ($\Omega,
\mathcal{F}, \{\mathcal{F}_t\}_{t\in\mathbb{R}}, \mathbb{P})$ with
$\mathcal F_t=\sigma(\{X_s, s\leq t\})$ denoting the completed filtration
generated by $X$. The law of $X$ is determined by its characteristic exponent
$\Psi$ which is the map $\Psi:\mathbb R\to\mathbb C$
that satisfies $\e[\te{\ii\theta X_1}]= \exp(\Psi(\theta))$.

The drawdown and drawup processes of $X$,
$(D_t)_{t\in\mathbb{R}_+}$
and $(U_t)_{t\in\mathbb{R}_+}$, are path-functionals of the increments of $X$
given by
$$
D_t =  \overline X_t - X_t,\q
U_t = X_t - \underline X_t,
$$
with $\overline X_t = \sup_{0\leq s\leq t} X_s$ and $\underline X_t = \inf_{0\leq s\leq t} X_s$.
We note that the drawdown $D_t$ and drawup $U_t$ at time $t$
are equal to the largest of all increments $X_{u} - X_t$, $u\in[0,t]$,
and the negative of the smallest increment of such increments.

Before turning to the analysis of the future drawdown and drawup extremes, we recall
a number of facts concerning drawup and drawdown processes which follow
from the fluctuation theory of L\'{e}vy processes. First of all, we note that
the marginal distributions of the drawup $U_t$ and
drawdown $D_t$, $t\in\mathbb R_+$,
can be expressed in terms of the marginal distributions of $X$ by deploying
the Wiener-Hopf factorisation of $X$, according to which the characteristic exponent $\Psi$
is related to the marginal distributions of the running supremum and
running infimum of $X$ at an exponential random time
$\mathbf e_q$ of parameter $q$ that is independent of $\mathcal F_\infty$ as follows:
$$
\frac{q}{q - \Psi(\theta)} = \e[\te{\ii\theta \ovl X_{\mathbf e_q}}]
\e[\te{\ii\theta \unl X_{\mathbf e_q}}],\qquad \theta\in\mathbb R, q\in\mathbb R_+\backslash\{0\}.
$$
Using the duality lemma
(see e.g. \cite[Proposition VI.3]{LP}) that
$U_t$ has the same law $\overline{X}_t$.
Thus
the Wiener-Hopf factorisation
may be phrased as follows in terms of the drawdown and drawup processes:
\begin{equation}\label{eq:wh}
\frac{q}{q - \Psi(\theta)} = \e[\te{\ii\theta U_{\mathbf e_q}}]
\e[\te{-\ii\theta D_{\mathbf e_q}}],\qquad \theta\in\mathbb R, q\in\mathbb R_+\backslash\{0\}.
\end{equation}
Moreover, since $U_t$ has the same law $\overline{X}_t$,
it follows that, if $\e[X_1]$ is strictly negative,
$U_t$ converges in distribution as $t\to\infty$ to a proper random variable
$U_\infty$ with the law of all-time supremum $\overline{X}_\infty$. Similarly,
if $\e[X_1]$ is strictly positive,
$D_t$ having the same law as $\underline{X}_t$ converges to a random variable $D_\infty$ as $t\rightarrow \infty$.
The Laplace transforms of $U_\infty$ and $D_\infty$ are given explicitly in terms of the
Laplace exponents $\kappa$ and $\WH\kappa$ of the ascending and descending ladder-height processes
$(L^{-1}, H)$ and $(\WH L^{-1}, \widehat H)$ . The ladder time process $L^{-1}~=~\{L^{-1}_t\}_{t\in\mathbb R_+}$ is
equal to the right-continuous inverse of a local time $L$ of
$(D_t)_{t\in\mbb R_+}$ at zero. The corresponding ladder-height
process $H=(H_t)_{t\geq 0}$ is given by $H_t = X(L^{-1}_t)$
for all $t\geq 0$ for which $L^{-1}_t$ is finite, and defined to
be $H_t=+\infty$ otherwise. We denote
$\kappa(\beta, \theta) =-\log \e[\exp\{-\beta L^{-1}_1-\theta
H_1\}\I_{\{H(1)<\infty\}}]$, where, for any set $A\in\mc F$, $\I_A$
denotes the indicator of the set $A$. Similarly, the Laplace exponent of the downward ladder-height process
$(\widehat{L}^{-1}, \widehat{H})$ corresponding
to the dual process $\widehat{X}$ of $X$, $\widehat{X}=-X$, we denote by
$\widehat{\kappa}(\beta, \theta) = -\log \e[\exp\{-\beta \widehat{L}^{-1}_1-\theta
\widehat{H}_1\}\I_{\{\widehat{H}(1)<\infty\}}]$. Specifically,
if $\e[X_1]$ is strictly positive, the Laplace transform of $D_\infty$ is
given as follows:
\begin{eqnarray}\label{eq:Uinfty}
&&\e[\te{-\theta D_\infty}] = \frac{\widehat{\kappa}(0,0)}{\widehat{\kappa}(0,\theta)};
\end{eqnarray}
see \cite{Kyprianou} for details.

A first step in the study of the random variables $\overline D^*_{t,s}$,
$\underline D^*_{t,s}$, $\overline U^*_{t,s}$ and $\underline U^*_{t,s}$ are the following distributional
identities.

\begin{prop}\label{repr} Let $t,s\in\mathbb R_+$
and let $\widetilde
U_s\ed U_s$ and $\widetilde
D_s\ed D_s$ be random variables independent of $\mathcal F_t$,
where $\ed$ denotes equality in distribution.
Denoting $\overline U_t=\sup_{0\leq u\leq t}U_u$,
$\overline D_t=\sup_{0\leq u\leq t}D_u$, we have the
following representations:
\begin{eqnarray}\label{eq:UOOL}
 \underline D^*_{t,s}& \ed &-
\max\left\{\widetilde D_s + D_{t},\overline D_t\right\},\qquad
\overline D^*_{t,s} \ed \min\left\{U_t- \WT D_s ,0\right\}
\end{eqnarray}
and
\begin{eqnarray}\label{eq:UOOLb}
\overline U^*_{t,s} &\ed&
\max\le\{\widetilde U_s + U_{t},\overline U_t\ri\},\qquad
\underline U^*_{t,s} \ed \max\{\widetilde U_s
- D_t,0\}.
\end{eqnarray}
In particular, when $\e[X_1]\in \mathbb R_+\backslash\{0\}$ $(\e[X_1]\in\mathbb R\backslash\mathbb R_+)$, then
$\overline D^*_t$ and $\underline D^*_t$ $(\overline U^*_t$ and $\underline U^*_t)$ are
finite $\p$-a.s.
\end{prop}
\begin{figure}[t]
\centering
\includegraphics[scale=0.24]{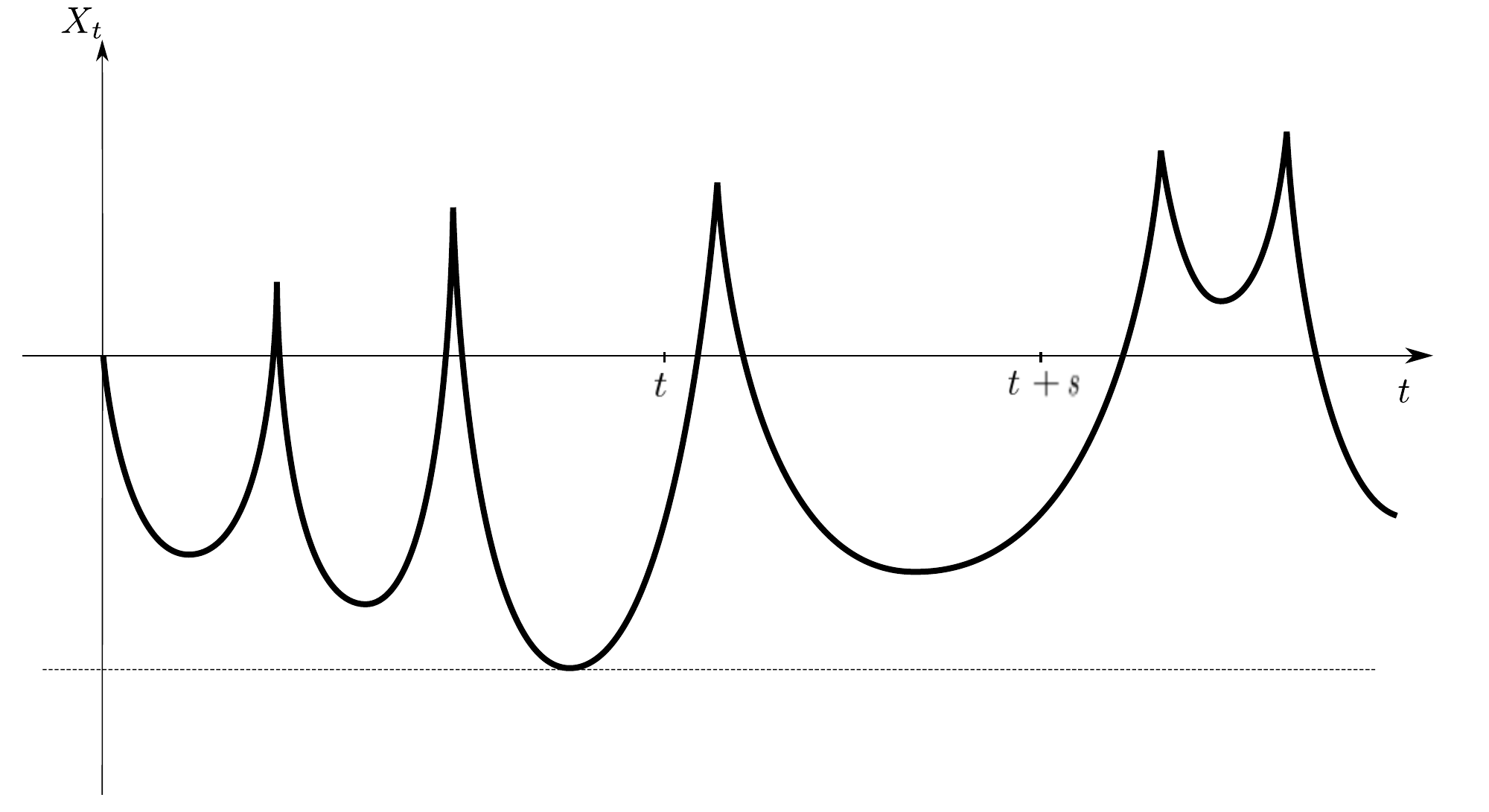}
\includegraphics[scale=0.24]{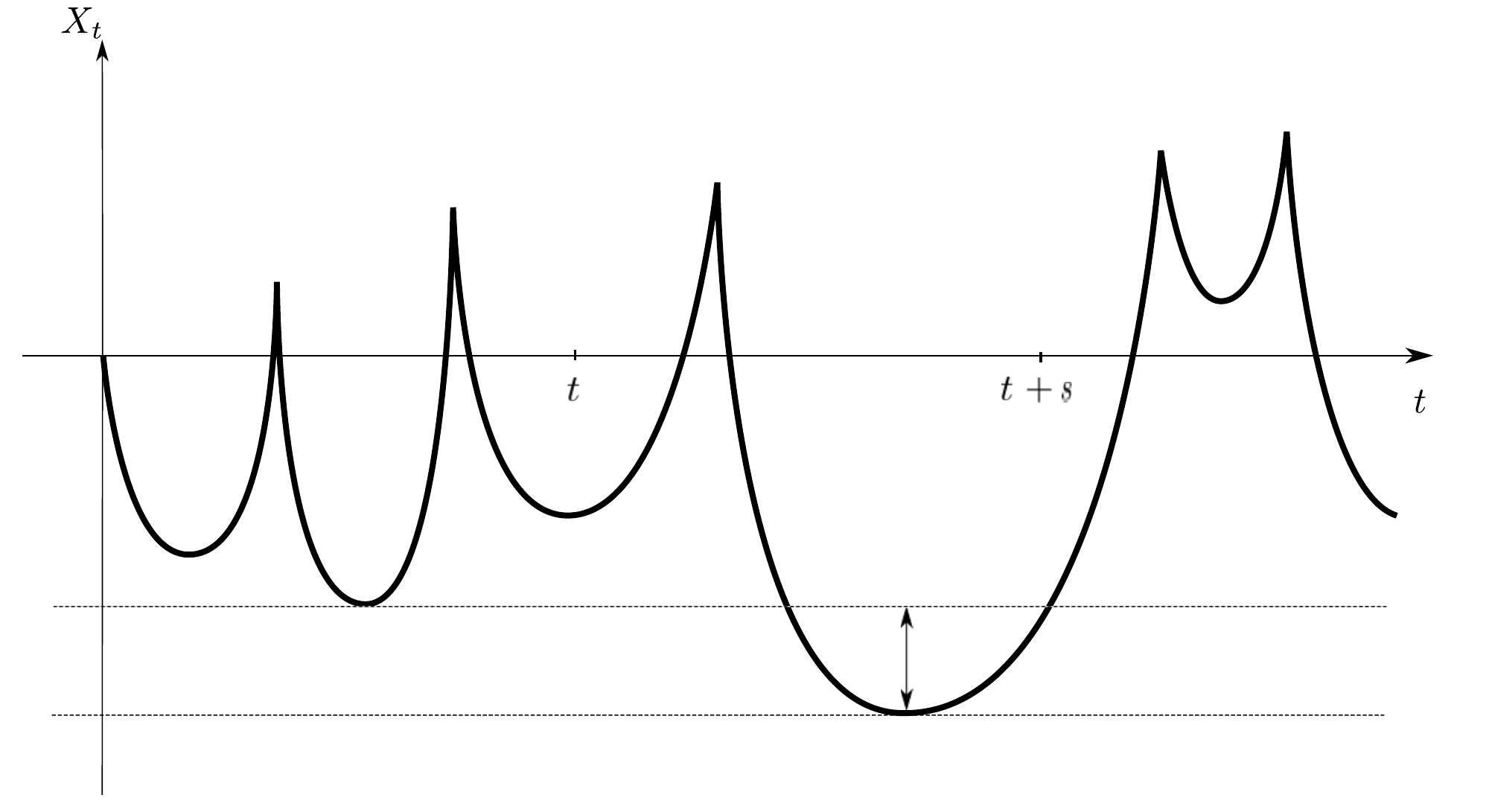}
\caption{\small Two schematic pictures of a part of the path of $X$
in the cases that (i) the smallest value of $X$ up to time $t+s$ has already been attained before time $t$
so that the path-functional $\underline D^*_{t,s}$ is zero (left-hand picture)
or (ii) $X$ attains a new minimum between $t$ and $t+s$
and the path-functional $\underline D^*_{t,s}$ is strictly negative (right-hand picture).
\label{fig:exc}}
 \end{figure}

\begin{rem}\label{dualextrema}
\rm
\begin{itemize}


\item[{\bf (i)}] Extending $X$ from $\mathbb R_+$ to a two-sided version on
$\mathbb R$ and using a time-reversal argument we find that
\begin{equation}
\overline U^*_t \ed\sup_{0\leq u\leq t}\sup_{-\infty<w\leq
u}(X_u-X_w),\qquad \underline U^*_t \ed\inf_{0\leq u\leq t}\sup_{-\infty<w\leq u}(X_u-X_w).
\end{equation}
Indeed, using the change of variables $u'=t-u$ and $w'=t-w$ we see
that
\begin{eqnarray*}
\sup_{0\leq u\leq t}\sup_{-\infty<w\leq u}(X_u-X_w)&=&
\sup_{0\leq u'\leq t}\sup_{w'\geq u'}(X_{t-u'}-X_{t-w'})\\
&\ed&\sup_{0\leq u'\leq t}\sup_{w'\geq u'}(X_{w'}-X_{u'}).
\end{eqnarray*}
The result for $\underline U^*_t$ follows similarly.

The random variables $\overline U^*_t$ and $\underline U^*_t$ arise in a queueing
application. Indeed, the workload process $Q_u$ of a queue with net input
process $X$ ({\em i.e.}, input less output) evolves according to the process $X$ reflected at its infimum,
{\em i.e.},
$Q_u=X_u-\inf_{s\leq u}X_s$. If we
assume that the workload process is stationary ({\em i.e.}, $Q_0$ follows
the stationary distribution, which is equal to the distribution of
$-\inf_{-\infty<s\leq 0}X_s$; see
\cite{reich}), then the workload $Q_u$ is given by:
\[Q_u=\sup_{-\infty <w\leq u} (X_u-X_w)\]
and $\overline U^*_t$ and $\underline U^*_t$
describe the supremum and infimum of the
workload process $Q$ over a finite time horizon $t$, respectively.
For details on queues driven by a L\'{e}vy process we refer to the survey book \cite{krzysiekbook}.

\item[{\bf (ii)}] We note $\mathbb{P}(\underline U^*_t=0)=\p(
\int_0^\infty \mathbf{1}_{(X_s\geq 0)}\ud s<t)$
(see for example \cite[Lemma 15, p. 170]{LP} and \cite[Theorem 13, p. 169]{LP}).
\end{itemize}
\end{rem}

\begin{proof}[Proof of Proposition \ref{repr}]
As noted in the Introduction, it suffices to establish the
statements concerning $\overline D^*$ and $\underline D^*$. Writing $[u,t+s]=[u,t]\cup
[t,t+s]$ for given $u,t,s\in\mbb R_+$ we have
\begin{eqnarray*}
\underline D^*_{t,s}&=&\inf_{0\leq u\leq t}
\min\le\{\inf_{w\in[t,t+s]}(X_w-X_t)+X_t-X_u,\inf_{u\leq w\leq t}(X_w-X_u)\ri\}.
\end{eqnarray*}
Since $\widetilde{D}_s :=-\inf_{t\leq w\leq t+s}(X_w-X_t)$ is independent of $\mathcal{F}_t$
and is equal in distribution to $D_s$, we find that $\underline D^*_{t,s}$ is
equal in distribution to
\begin{eqnarray*}
\inf_{0\leq u\leq t}\min\le\{X_{t}-X_u -\widetilde{D}_s,\inf_{u\leq
w\leq t}(X_w-X_u)\ri\} &=&
\min\le\{-D_t-\widetilde D_s,\inf_{0\leq
u\leq t}\inf_{u\leq
w\leq t}(X_w-X_u)\ri\}\\
&=&-\max\le\{D_t+\widetilde D_s,
\sup_{0\leq w\leq t}\sup_{0\leq
u\leq w}(X_u-X_w)\ri\},
\end{eqnarray*}
which yields the first identity in \eqref{eq:UOOL}.

For the second identity in \eqref{eq:UOOL} we note that
the function $u\mapsto \inf_{u\leq w\leq t+s}(X_w-X_u)$ attains its supremum over $[0,t]$ at $G_{t-}$ or
$G_t$ where $G_t=\sup\{u\leq t: X_u= \underline{X}_t\}$.
In the case that $G_{t+s}\leq t$ ({\em i.e.}, when $\underline{X}_{t+s}=\underline{X}_t$) we have $G_t = G_{t+s}$
(see Figure~\ref{fig:exc}, left-hand picture)
and
$\overline D^*_{t,s} = 0,$
while in the case that $G_{t+s}>t$ (see Figure~\ref{fig:exc}, right-hand picture) we find
$$\overline D^*_{t,s}= \underline{X}_{t+s}-\underline{X}_t<0.$$ Hence, writing
$\underline{X}_{t+s} = \min\{\inf_{t\leq u\leq t+s}(X_u-X_t) + X_t,\underline{X}_t\}$
we deduce that
$$\overline D^*_{t,s}\ed\min\le\{X_t-\underline X_t+\inf_{0\leq w\leq s}\widetilde{X}_w,0\ri\},$$
where $\widetilde{X}$ denotes an independent copy of $X$, from which the expression for $\overline{D}_{t,s}^*$ follows.

Taking $s\rightarrow\infty$ in \eqref{eq:UOOL} and noting that $-\inf_{s\geq 0}X_s$ is  finite $\p$-a.s.
if $\e[X_1]\in\mathbb R_+\backslash\{0\}$ we conclude that also $\overline D^*_t$ and $\underline D^*_t$ are $\p$-a.s. finite.
 \end{proof}
\bigskip

\section{Asymptotic future drawdown --- the light-tailed case}\label{sec: ascramer}

In this section we study the asymptotics of the tail
probabilities $\p(-\overline D^*_t > x)$ and $\p(-\underline
D^*_t>x)$ in the case that the L\'{e}vy measure is light-tailed.
More specifically, in this section we will make the following assumptions.

\begin{As}\label{A2}The Cram\'{e}r condition holds, {\em i.e.},
\begin{equation}\label{cr}
\text{there exists a $\gamma\in\mathbb R_+\backslash\{0\}$ satisfying $\e[\te{-\gamma
X_1}]=1$,}
\end{equation}
The mean of $X_1$ is positive and finite, $\e[X_1]\in\mathbb R_+\backslash\{0\}$,
and $\e[\te{-\gamma X_1}|X_1|]\in\mathbb R_+\backslash\{0\}$.
\end{As}
\begin{As}\label{A3}
$X$ has non-monotone paths and either
$0$ is regular for $\mathbb R_+\backslash\{0\}$ or the L\'{e}vy measure of $X$
is non-lattice.
\end{As}
Under condition \eqref{cr} the characteristic exponent
$\Psi$
can be extended to the strip
$\mathcal S_\gamma = \{\theta\in\mathbb C: \Im(\theta) \in [0,\gamma]\}$ of the complex plane,
by analytical continuation and continuous extension.
The Laplace exponent $\psi(\theta) = \log \e[\te{\theta X_1}]$ of $X$
is finite on the maximal domain
$\Theta = \{\theta\in\mathbb R: \psi(\theta)<\infty\}$, which contains
the interval $[-\gamma,0]$.
Restricted to the
interior $\Theta^o$, the map $\th\mapsto \psi(\th)$ is convex and
differentiable, with derivative $\psi'(\th)$.\footnote{For
$\th\in\Theta\backslash\Theta^o$, $\psi'(\theta)$ is understood to
be $\lim_{\eta\to\theta,\eta\in\Theta^o}\psi'(\eta)$. }

Under \eqref{cr} the Wiener--Hopf factorisation
\eqref{eq:wh} remains valid for $\theta$ in the strip $\mc S_\gamma$.

\begin{lemma}\label{lem:wh}
If Assumption \ref{A2} is satisfied, we have
\begin{equation}\label{eq:whg2}
\e[\te{\gamma D_{\mathbf e_q}}]  < \infty.
\end{equation}
\end{lemma}
\begin{proof} It follows from the Wiener--Hopf factorisation \eqref{eq:wh} that
\begin{equation}\label{star}
\e[\te{-\ii\theta D_{\mathbf e_q}}] = q (q - \Psi(\theta))^{-1}\e[\te{\ii\theta U_{\mathbf e_q}}]^{-1}
\end{equation}
for all $\theta$ in the interior of the strip $\mathcal S_\gamma$.
We note  that $\e[\te{\ii\theta U_{\mathbf e_q}}]$ is continuous and strictly positive
on the set $\mathcal A = \{\theta: -\ii\theta\in[0,\gamma]\}$. Moreover, $\Psi(\theta)$
can be analytically extended to $\mathcal A$. Indeed, note that $\Psi(\theta)=\Psi_1(\theta)+\Psi_2(\theta)$
where $\Psi_1(\theta)$ is entire function by \cite[Lem. 25.6, p. 160]{Sato} and
$\Psi_2(\theta)=\int_{|x|>1} e^{-\gamma x}\; \widehat{{\mc V}}(\ud x)$
is finite by Assumption \ref{A2} and \cite[Thm. 3.6, p. 76]{Kyprianou}
for a L\'evy measure $\widehat{{\mc V}}$ of $X$.
This, combined with the fact $\Psi(\ii\gamma)=0$, yields \eqref{eq:whg2}.
\end{proof}

In \cite{BertDon} it was shown that under Assumptions \ref{A2} and \ref{A3},
Cram\'{e}r's estimate holds for the  L\'{e}vy process $X$, {\em i.e.},
\begin{equation}\label{eq:cra}
\p(D_\infty > y) \simeq C_\g\te{-\gamma y}, \qquad C_\g=
\frac{\widehat{\kappa}(0,0)}{\gamma \left[\frac{\partial}{\partial \theta}\widehat{\kappa}(0,-\theta)\right]_{|\theta=\gamma}}>0, \qquad \text{as
$y\to\infty$},
\end{equation}
where
we write
$f(x) \simeq g(x)$ as $x\to\infty$ if
$\lim_{x\to\infty}f(x)/g(x)=1$. Cram\'{e}r's estimate can be extended
to the decay of the finite time probability $\p(D_s>x)$
when $x,s$ jointly tend to infinity
in some fixed proportion, that is when we have $x=vs + {\rm o}(s^{1/2})$.
The proportions $v$ are to be positive and lie in the range of $\psi'$. This leads to the following definition.
\begin{defi} A proportion $v\in\mathbb R_+\backslash\{0\}$ is {\em feasible} if
there exists a $\xi_v\in\Theta^o$
such that $\psi'(\xi_v) = -v$.
\end{defi}
More specifically, it was shown in \cite{PP} that
if the proportion $v$ is feasible and satisfies $0< v < -\psi'(-\gamma)$
 the H\"{o}glund's estimates hold for $X$, {\em i.e.}, if Assumptions \ref{A2} and \ref{A3} are
satisfied, then for $x$ and $s$ tending to infinity such that $x=vs + {\rm o}(s^{1/2})$
we have
\begin{eqnarray}\label{eq:hog}
&&  \p(D_s > x) \sim C_\gamma\te{-\gamma x},
\end{eqnarray}
where we write $f\sim g$ if $\lim_{x,s\to\infty, x=vs+{\rm o}(s^{1/2})}
f(x,s)/g(x,s) = 1$.

Using the representations in Proposition~\ref{repr} we identify the exact asymptotic decay
of the tail probabilities of $\underline D^*_{t,s}$ and $\overline D^*_{t,s}$ as follows:

\begin{theorem}\label{Cramer}
Suppose that Assumptions \ref{A2} and \ref{A3} hold, and let $t\in\mathbb R_+\backslash\{0\}$.

{\bf (i)} Then the following limit hold true:
\begin{equation}\label{alim}
\p(-\underline D^*_t>x) \simeq
C_\gamma\e[\te{\gamma D_t}]\, \te{-\gamma x},\quad x\to\infty
\end{equation}
and
\begin{equation}\label{alim2}
\p(-\overline D^*_t>x) \simeq
C_\gamma\e[\te{-\gamma U_t}]\, \te{-\gamma x}\quad x\to\infty.
\end{equation}

{\bf (ii)} Let $0<v<-\psi'(-\gamma)$.
If $x$ and $s$ tend to infinity
such that $x=vs + {\rm o}(s^{1/2})$ for some feasible proportion $v$
then we have the following limits:
\begin{eqnarray}\label{blim}
\p(-\overline D^*_{t,s} >x)&\sim&
C_\gamma\e[\te{-\gamma U_t}]\,
\te{-\gamma x},\\
\p(-\underline D^*_{t,s}>x) &\sim&
C_\gamma\e[\te{\gamma D_t}]\, \te{-\gamma x}.
\label{clim}
\end{eqnarray}
\end{theorem}
\begin{rem}\rm In specific cases the Wiener--Hopf factors
are known in explicit analytical form, so that the constants in \eqref{alim} can be identified.
\begin{itemize}
\item[{\bf (i)}] If $X$ is spectrally positive, then $C_\gamma = 1$ and
\begin{equation}\label{specnegprefactors}
\e[\te{\gamma D_{\mathbf e_q}}] =  \frac{\widehat{\Phi}(q)}{\widehat{\Phi}(q)-\gamma},\q q>0,\end{equation} where
$\gamma=\widehat{\Phi}(0)$,  with $\widehat{\Phi}(q)$, $q\geq 0$, the largest root of
the equation $\widehat{\psi}(\theta)=q$ where $\widehat{\psi}(\theta) = \log
\e[\te{-\theta X_1}]$ is the Laplace exponent of the dual process $\widehat{X}=-X$.
These expressions
hold since $D_{\mathbf e_q}$ has the same law $\widehat{\overline{X}}_{\mathbf e_q}$ and hence
follows an exponential distribution with parameter $\widehat{\Phi}(q)$. By inverting the Laplace
transforms in $q$ we find the following explicit expression in terms
of the one-dimensional distributions of $X$:
\begin{equation}\label{spnegUT}
\e[\te{\gamma D_t}] = 1+\gamma\int_0^t \mathbb{E}[\te{-\gamma X_z}
X_z^-]z^{-1}\ud z,
\end{equation}
where $X_t^-=\min\{X_t,0\}$.
Indeed, note that
$\e[\te{\gamma D_t}]=\e[\te{\gamma \widehat{U}_t}]$.
Moreover, on account of Kendall's identity $(\mathbb{P}(\tau_x^+\in\ud t) = \frac{x}{t}\p(\widehat{X}_t\in \ud x)$ for
$x,t\in\mathbb R_+\backslash\{0\}$ and the first passage time $\tau_x^+=\inf\{t\geq 0: \widehat{X}_t>x\}$), it follows that
\begin{equation}
\int_0^\infty \te{-qt}\mathbb{E}[\te{-\gamma \widehat{X}_t} \widehat{X}_t^+]t^{-1}\ud t
=\frac{1}{\widehat{\Phi}(q)+\gamma},\label{invLTmain}
\end{equation}
where $\widehat{X}_t^+=\max\{\widehat{X}_t^+,0\}$.
Further, from \cite[eq. (8.2)]{Kyprianou} and fact that $\widehat{\psi}(\gamma)=\psi(-\gamma)=0$,
\begin{equation}\label{specposcalki2}
\e[\te{-\gamma U_{\mathbf e_q}}] =\e[\te{\gamma \widehat{\underline{X}}_{\mathbf e_q}}]=
\frac{q}{q - \widehat{\psi}(\gamma)} \left[1 - \frac{\gamma}{\widehat{\Phi}(q)}\right]=1 - \frac{\gamma}{\widehat{\Phi}(q)}.
\end{equation}
Hence, we have
$$
\e[\te{-\gamma U_t}]=1-\gamma\int_0^t
\e [X^-_z]z^{-1}\,\ud z.
$$

\item[{\bf (ii)}] If $X$ is spectrally negative, then we have $C_\gamma =
\frac{\psi^\prime(0)}{|\psi^\prime(-\gamma)|}$ and
\begin{equation}\label{specposprefactors}
\e[\te{-\gamma U_{\mathbf e_q}}]^{-1} = \e[\te{\gamma D_{\mathbf e_q}}] = \frac{\Phi(q)+\gamma}{\Phi(q)},
\end{equation}
where $\gamma$ and $\Phi(q)$, $q\geq 0$, are
the largest roots of $\psi(-\theta)=0$ and
$\psi(\theta)=q$ for the Laplace exponent $\psi(\theta) = \log \mathbb{E}[\te{\theta X_1}]$.
Hence
\begin{equation}\label{asref}
\e[\te{\gamma D_t}] = 1+\gamma\int_0^t
\mathbb{E}[{X}_z^+]z^{-1}\ud z,\qquad
\e[\te{-\gamma U_t}]=1-\gamma\int_0^t\e
[\te{-\gamma X_z}X^+_z]z^{-1}\,\ud z.
\end{equation}

\item[{\bf (iii)}] The Wiener--Hopf factors may also be identified for the meromorphic L\'{e}vy processes \cite[Def. 1]{Mero}:
\[\e[\te{\gamma U_{\eq}}]=\prod_{n\geq 1}\frac{1-\frac{\gamma}{\rho_n}}{1-\frac{\gamma}{\zeta_n(q)}}, \qquad\e[\te{\gamma D_{\eq}}]=
\prod_{n\geq 1}\frac{1-\frac{\gamma}{\hat{\rho}_n}}{1-\frac{\gamma}{\hat{\zeta}_n(q)}},\]
where $\{-\ii\rho_n,\ii\hat{\rho}_n\}_{n\geq 1}$ are the poles of $\Psi$ (which is meromorphic) and $\{-\ii\zeta_n(q),\ii\hat{\zeta}_n(q)\}_{n\geq 1}$ are the roots of $q+\Psi(\theta)=0$.
The above Laplace transforms in $q$ can be numerically inverted giving
$\e[\te{\gamma U_t}]$ and
$\e[\te{\gamma D_t}]$ (see for details \cite[Sec. 8]{Mero}).
\end{itemize}
\end{rem}

\begin{proof}[Proof of Theorem \ref{Cramer}]
{\bf (i)} From Proposition \ref{repr} it follows that for $s,t\in\mbb R_+$,
\begin{multline}
\p(-\underline D^*_{t,s}\leq x)=\int_{[0,x]}\p(D_s\leq x-z)\p(D_t\in\ud z,\overline{D}_t\leq x)
\Leftrightarrow \\
 \p(-\underline D^*_{t,s}> x) = \p(\overline D_t> x) +
\int_{[0,x]}\p(D_s>x-z)\p(D_t\in\ud z,\overline{D}_t\leq x).\label{Ustarr}
\end{multline}
By letting $s\to\infty$ in \eqref{Ustarr}
we arrive at
the identity
\begin{equation}\label{eq:asympp}
\p(-\underline D^*_t>x)=\p(\overline{D}_t>x)+\int_{[0,x]}\p(D_\infty>x-z)\p(D_t\in\ud z,\overline{D}_t\leq x).\end{equation}
Denote by $\p^{(\gamma)}$ the Cram\'{e}r measure which is defined on $(\Omega,\mathcal F_t)$ by $\p^{(\gamma)}(A) = \e[\te{-\gamma X_t}\mathbf 1_A]$, $A\in\mathcal F_t$.
The Cram\'{e}r asymptotic decay \eqref{eq:cra} implies that
\begin{equation}\label{eq:cl}
\te{\gamma x}\p(D_\infty > x) = \e^{(\gamma)}[\te{-\gamma(\widehat{X}_{\tau_x^+} - x)}]
\simeq C_\gamma, \quad \text{as $x\to\infty$}.
\end{equation}
In view of the facts that $t\mapsto \overline X_t$ is non-decreasing and $D_{T^D_x}-x\geq 0$ for
$T_x^D=\inf\{t\geq 0: D_t>x\}$ and any $x\in\mathbb R_+\backslash\{0\}$,
we find\footnote{$f(x)=o(g(x))$ for $x\to\infty$ if $|f(x)/g(x)|\to 0$ as $x\to\infty$.}
\begin{eqnarray}
\nn \p(\overline{D}_t>x)&=&\p(T^D_x<t)=\te{-\gamma x}\e^{(\gamma)}[\te{\gamma(X_{T^D_x}+x)}\I_{\{T^D_x<t\}}]\\
\nn &=&\te{-\gamma x}\e^{(\gamma)}[\te{-\gamma(D_{T^D_x}-x-\overline X_{T^D_x})}\I_{\{T^D_x<t\}}]\\
&\leq & \te{-\gamma x}\e^{(\gamma)}[\te{\gamma \overline X_{t}}\I_{\{T^D_x<t\}}]
 ={\rm o}(\te{-\gamma x}),\quad\mbox{as }x\rightarrow\infty,
\label{oegx}
\end{eqnarray}
where the expectation in \eqref{oegx} converges to zero by virtue of the dominated convergence theorem and the facts that
$\e^{(\gamma)}[\te{\gamma\overline X_{t}}] <\infty$ (by Lemma \ref{lem:wh})
and $T^D_x\rightarrow \infty$ $\p^{(\gamma)}$-a.s. as $x\rightarrow\infty$ (as $X_t\to-\infty$ as $t\to\infty$, $\p^{(\gamma)}$-a.s.).
Combining \eqref{eq:asympp} with
 \eqref{oegx}, the Cram\'{e}r asymptotics \eqref{eq:cl}  and
 the dominated convergence theorem yield
\[\lim_{x\rightarrow\infty}\te{\gamma x}\p(-\underline{D}^*_t>x)
=C_\gamma\int_{\mathbb R_+}\te{\gamma z}\p(D_t\in\ud z)=C_\gamma \e[\te{\gamma D_t}], \q t\in\mbb R_+.\]

As far as $\overline D^*_t$ is concerned,
we deduce from Proposition~\ref{repr}, the Cram\'er asymptotics \eqref{eq:cra}, Lemma~\ref{lem:wh}
and the dominated convergence theorem that
\begin{eqnarray}
\p(\overline D^*_t>x)
&=&\int_{\mathbb R_+}\p(D_\infty>x+z)\p(U_t\in\ud z)\label{reprunderUstar}\\ \nn
&\simeq& C_\gamma \te{-\gamma x}\int_{\mathbb R_+} \te{-\gamma z}
\p(U_t\in \ud z) = C_\gamma \te{-\gamma x}\e[\te{-\gamma U_t}].
\end{eqnarray}

{\bf (ii)} Let $v$ be a feasible proportion.
The proof follows by a line of reasoning that is analogous to the one given in part (i),
deploying H\"{o}glund's estimate \eqref{eq:hog} instead of Cram\'{e}r's estimate. In particular,
combining \eqref{eq:hog}, \eqref{Ustarr}, \eqref{oegx} and
the dominated convergence theorem shows that when $0<v<-\psi'(-\gamma)$
\begin{eqnarray*}
&& \te{\gamma x}\p(-\underline{D}^*_{t,s}>x) \sim C_\gamma\int_{[0,\infty)}\te{\gamma z}\p(D_t\in\ud z)=
C_\gamma  \e[\te{\gamma D_t}].
\end{eqnarray*}

\end{proof}

\subsection{Asymptotic drawdown and drawup measures}\label{sec: asmeasures}
Conditional on $-\overline D_{t,s}^*$ being large, for fixed $s,t\in\mathbb R_+$,
or on $-\underline D_{t,s}^*$ being large, $X_t$ admits a limit in
distribution, as we show next.
These limits are given by the ``drawup-measures'' $\overline{\p}^{(s)}$ and
the ``drawdown measures'' $\underline{\p}^{(s)}$, $s\in\Theta$, that are defined as follows
on the
measurable space $(\Omega,\mathcal F_t)$:
\begin{eqnarray}
&& \overline\p^{(s)}(A) = \e\le[\frac{\te{-s U_t}}{\e[\te{-s U_t}]}
\mathbf 1_A\ri],\q
 \underline\p^{(s)}(A) = \e\le[\frac{\te{s D_t}}{\e[\te{s D_t}]}
\mathbf 1_A\ri], \qquad A\in\mathcal F_t.\label{miara}
\end{eqnarray}

\begin{cor}\label{thm:asymp}
Suppose Assumptions \ref{A2} and \ref{A3} hold, and let $t\in\mathbb R_+\backslash\{0\}$.

{\bf (i)} Then,
conditional on $\{\overline D^*_t < -x\}$
and on $\{\underline D^*_t < -x\}$,
 $X_t$ converges in distribution
as $x\to\infty$:
\begin{eqnarray}\label{clim1}
\p[X_t\leq x| -\underline D^*_t > x] &\simeq&
\underline \p^{(\gamma)}[X_t\leq x],\\
\p[X_t \leq x| -\overline D^*_t > x] &\simeq&
\overline \p^{(\gamma)}[X_t \leq x].\label{clim2}
\end{eqnarray}

{\bf (ii)} Let $0<v<-\psi'(-\gamma)$. If $x$ and $s$ tend to infinity
such that $x=vs + {\rm o}(s^{1/2})$ where $v$ is feasible
then the following limits hold true:
\begin{eqnarray}\label{clim11}
\p[X_t\leq x| -\underline D^*_{t,s} > x] &\sim&
 \underline \p^{(\gamma)}[X_t\leq x],\\
\p[X_t \leq x| -\overline D^*_{t,s} > x] &\sim&
\overline \p^{(\gamma)}[X_t\leq x].
\end{eqnarray}
\end{cor}
\begin{proof}[Proof of Corollary \ref{thm:asymp}]
{\bf (i)} By following a similar line of reasoning as the proof of Theorem~\ref{Cramer}
it is straightforward to show that for $\theta \in[0, \gamma]$, as $x\to\infty$,
\begin{eqnarray*}
\e[\te{\theta X_t}\mathbf 1_{\{-\underline D^*_t>x\}}] &\simeq&
C_\gamma \te{-\gamma x}\e[\te{\theta X_t + \gamma D_t}],\\
\e[\te{\theta X_t}\mathbf 1_{\{-\overline D^*_t>x\}}] &\simeq&
C_\gamma \te{-\gamma x}\e[\te{\theta X_t -\gamma U_t}].
\end{eqnarray*}
Bayes' lemma then yields the stated identities.
The proof of (ii) is similar and is omitted.
\end{proof}

\section{Asymptotic future drawdown --- the heavy-tailed case}\label{sec:heavy}

We continue the study of the asymptotic behaviour of the tail probabilities
of $\overline D^*_t$ and $\underline D^*_t$ in the case that the L\'{e}vy measure
$\mc V$ of $\widehat{X}=-X$ belongs to  the class $\mathcal{S}^{(\alpha)}$ of convolution-equivalent measures
which, we recall, is a subset of the class $\mathcal{L}^{(\alpha)}$ defined as follows.

\begin{defi}\label{def3}(Class $\mathcal{L}^{(\alpha)}$)
For a parameter $\alpha \in\mathbb R_+$ we say that measure
$G$ with tail $\overline{G}(u):=G((u,\infty))$ belongs to class
$\mathcal{L}^{(\alpha)}$ if
\begin{itemize}
\item[{\bf (i)}]  $\overline{G}(u)>0$ for each $u\in\mathbb R_+$,
\item[{\bf(ii)}] $\lim_{u \rightarrow \infty}
\frac{\overline{G}(u-y)}{\overline{G}(u)}=\te{\alpha y} \textrm{
for each $y \in \mathbb R$, and $G$ is nonlattice}$, \item[{\bf(iii)}]
$\lim_{n \rightarrow \infty}
\frac{\overline{G}(n-1)}{\overline{G}(n)}=\te{\alpha} \textrm{   if
$G$ is lattice}$ (then assumed of span $1$).
\end{itemize}
\end{defi}

\begin{defi} (Class $\mathcal{S}^{(\alpha)}$)
We say that $G$ belongs to class $\mathcal{S}^{(\alpha)}$ if
\begin{itemize}
\item[{\bf(i)}] $G \in \mathcal{L}^{(\alpha)}$; \item[{\bf(ii)}] for
some $M_0\in\mathbb R_+$, we have
\begin{eqnarray}
\lim_{u \rightarrow \infty}
\frac{\overline{G^{*2}}(u)}{\overline{G}(u)}=2M_0,
\end{eqnarray}
where $\overline{G^{*2}}(u)=G^{*2}(u,\infty)$ and $*$ denotes
convolution.
\end{itemize}
\end{defi}

The asymptotics are derived under conditions on the L\'evy measure $\Pi$ of the downward ladder height process $\widehat{H}$,
which according to the Vigon \cite{Vigon} identity is related to the L\'evy measures $\mc V$ of $\widehat{X}$ by
$$\overline{\Pi}(z) = {\Pi}((z,\infty)) =-\int_{\mathbb R\backslash\mathbb R_+}\overline{\mc V}(u-y)V(\ud y), \q z\in\mathbb R_+,$$
for the renewal measure $V(\ud y)=\int_0^\infty \p(H_t\in \ud y)\ud t$ and $\overline{\mc V}(y)={\mc V}(y,\infty)$.
Throughout this section we assume that for some fixed $\alpha\in\mathbb R_+\backslash\{0\}$ the following three conditions hold true:
\begin{eqnarray}\label{Con1}
&&
\overline{\Pi} \in\mathcal{S}^{(\alpha)};\\
\label{Con2b} &&\widehat{\psi}(\alpha)=\psi(-\alpha)\in\mathbb R\backslash\mathbb R_+;\\
\label{Con2} &&\widehat{\kappa}(0,0)+\widehat{\kappa}(0,-\alpha)\in\mathbb R_+\backslash\{0\}.
\end{eqnarray}

\begin{theorem}\label{Heavy}
Assume that $\e[X_1]\in\mathbb R_+\backslash\{0\}$ and let $t\in\mathbb R_+\backslash\{0\}$. Under conditions \eqref{Con1}--\eqref{Con2} we have:
\begin{eqnarray*}
\p(-\underline D^*_t > x) \simeq {\rm const}^+_t \overline{\Pi}(x), \qquad \p(-\overline D^*_t > x) \simeq
{\rm const}^-_t \overline{\Pi}(x),
\end{eqnarray*}
where functions ${\rm const}^+_t$ and ${\rm const}^-_t\in\mathbb R_+$ are given by
\begin{equation}\label{constantplus}
{\rm const}^+_t=\e[\te{\alpha \overline{\widehat{X}}_{t}}]+ \int_{[0,t]} \e\left[\te{\alpha \underline{\widehat{X}}_{t-z}}\right]^{-1}\mu(\td z)
=\e[\te{-\alpha \underline{X}_{t}}]+ \int_{[0,t]} \e\left[\te{-\alpha \overline{X}_{t-z}}\right]^{-1} \mu(\td z),
\end{equation}
and
\begin{equation*}
\q {\rm const}^-_t= \e[\te{-\alpha \underline{\widehat{X}}_{t}}]=\e[\te{\alpha \overline{X}_{t}}],
\end{equation*}
with the Borel measure $\mu$ on $(\mathbb R_+, \mathcal B(\mathbb R_+))$ given by
\begin{equation}\label{F2}
\mu({\mathrm d} z)=\int_{0}^\infty \p(\widehat{L}^{-1}_m\in {\rm d}z)
\te{-\widehat{\kappa} (0,-\alpha) m}\left[1-m\widehat{\kappa}(0,-\alpha)\right]\ud m.
\end{equation}
\end{theorem}
\begin{rem}\label{remmq}\rm
\begin{itemize}
\item[{\bf (i)}] By straightforward calculations it can be verified that
\begin{equation}\label{mq}(\mathcal L\mu)(q) = \frac{1}{q}\cdot \frac{\widehat{\kappa}(q,0)}{(\widehat{\kappa}(q,0)+\widehat{\kappa}(0,-\alpha))^2},
\end{equation}
where $\mathcal L\mu$ denotes the Laplace-Stieltjes transform of the
measure $\mu$.

\item[{\bf (ii)}] If $\mc V\in \mathcal{S}^{(\alpha)}$ for $\alpha >0$ then \eqref{Con1} holds and
\[\overline{\Pi}(x)\simeq \frac{1}{\kappa(0, -\alpha)}\overline{\mc V}(x);\]
see \cite[Proposition 5.3]{KyprKlup}.

\item[{\bf (iii)}] If $X$ is spectrally positive, then from \eqref{spnegUT} and \eqref{specposcalki2}:
\begin{equation}
\e[\te{-\alpha \underline{X}_{t}}] = 1+\alpha\int_0^t
\mathbb{E}[\te{-\alpha X_z}X_z^-]z^{-1}\ud z,\qquad
\e[\te{\pm\alpha \overline{X}_{t}}] = \te{t\psi(\mp \alpha)}\pm\alpha\int_0^t\te{(t-z)\psi(\mp\alpha)}\e
X^-_zz^{-1}\,\ud z.
\end{equation}
Moreover, since $\widehat{\kappa}(q,0)=q/\widehat{\Phi}(q)$ and $\widehat{\kappa}(0,-\alpha)=-\widehat{\psi}(\alpha)/(\widehat{\Phi}(0) +\alpha)$, we have
\[q\, (\mathcal L\mu)(q) = \frac{\widehat{\kappa}(q,0)}{(\widehat{\kappa}(q,0)+\widehat{\kappa}(0,-\alpha))^2}= \frac{q\widehat{\Phi}(q)}{(q-\widehat{\Phi}(q)^2\frac{\widehat{\psi}(\alpha)}{\widehat{\Phi}(0)+\alpha})^2}.\]
\end{itemize}
\end{rem}
\begin{proof}[Proof of Theorem \ref{Heavy}]
We first prove the statement concerning $\underline D^*_t$.
The starting point of the proof is to take the identity noted earlier in \eqref{eq:asympp}
and replace the fixed time $t$  by an independent exponential random variable $\eq$ with parameter $q$,
which yields
\begin{equation}\label{basicheavyover}
\p(-\underline D^*_{\eq}>x)=\p(\overline{D}_{\eq}>x)+\int_{[0,x]}\p(D_\infty>x-z)\p(D_{\eq}\in\ud z,\overline{D}_{\eq}\leq x).
\end{equation}
We show that both terms on the right-hand side of \eqref{basicheavyover} are asymptotically equivalent to the tail-measure
$\overline{\Pi}(x)$ of the ladder process $\widehat{H}$ as $x\to\infty$ and identify the constant.
As before we denote the first
upward and downward passage times of $\widehat{X}$ across the level $x$ by
$\tau_{x}^+=\inf\{t\geq 0: \widehat{X}_t>x\}$ and $\tau_{x}^-=\inf\{t\geq 0: \widehat{X}_t<x\}$.

To establish this result it suffices to show asymptotic equivalence of the two terms on the right-hand side of \eqref{basicheavyover}
to the probability $\p(\tau_x^+<\eq)$, since it is known from \cite[Theorem 4.1]{KyprKlup} and \cite[Lemma 5.4, eq. (5.6)]{jamaria} that under the conditions stated in the theorem
 \begin{equation}\label{jamaria2}
\p(\tau_x^+<\eq) \simeq
\frac{\widehat{\kappa}(q,0)}{(\widehat{\kappa}(q,0)+\widehat{\kappa}(0,-\alpha))^2}\cdot \overline{\Pi}(x),
\qquad q\ge 0,
\end{equation}
with the interpretation $\p(\tau_x^+<\infty)=\p(\tau_x^+<\mathbf{e}_0)$ for $q=0$.
Note that the constant in \eqref{jamaria2} is strictly positive for all $q\ge 0$
by the condition \eqref{Con2} and $\widehat{\kappa}(0,0)>0$ (as $\e[\widehat{X}_1]$ is strictly negative by the assumption that $\e[X_1]>0$).

We treat both terms separately, starting with the first term. We first derive
upper and lower bounds for the ratio $\p(\overline{D}_{\eq}>x)/\p(\tau_x^+<\eq)$.
By an application of the strong Markov property and the definition of $\ovl U$  we have
\begin{eqnarray}\nonumber
\p(\overline{D}_{\eq}>x) &\geq&
\p(\tau_x^+<\tau_{-\epsilon}^-\wedge\eq, \overline{D}_{\eq}>x)+\p(\tau_{-\epsilon}^-<\tau_x^+\wedge\eq, \overline{D}_{\eq}>x)\\
&=& \p(\tau_x^+<\tau_{-\epsilon}^-\wedge\eq)+\p(\tau_{-\epsilon}^-<\tau_x^+\wedge\eq)\p(\overline{D}_{\eq}>x)\q\text{and}
\label{lower}
\\ \nonumber
\p(\overline{D}_{\eq}>x)
&\leq& \p(\tau_x^+<\tau_{-\epsilon}^-\wedge\eq)
+\p(\tau_{-\epsilon}^-<\tau_x^+\wedge\eq)\p(\overline{D}_{\eq}>x) + A_q
\q\text{with}\\
A_q &=& \p(\underline{\widehat{X}}_{\eq}>-\epsilon, x+\epsilon \geq
\widehat{X}_{\eq} -\underline{\widehat{X}}_{\eq} \geq x) =
\p(\underline{\widehat{X}}_{\eq}>-\epsilon)\p( x+\epsilon \geq
\overline{\widehat{X}}_{\eq} \geq x),\label{upper}
\end{eqnarray}
where in the last line we used that $\underline{\widehat{X}}_{\eq}$ and $\widehat{X}_{\eq} -\underline{\widehat{X}}_{\eq}$ are independent
(by the Wiener--Hopf factorisation) and $\widehat{X}_{\eq} -\underline{\widehat{X}}_{\eq}$ and $\overline{\widehat{X}}_{\eq}$ have the same distribution.
Hence we find from \eqref{lower} and \eqref{upper} that
\begin{eqnarray}\label{upperdwa}
\frac{\p(\overline{D}_{\eq}>x)}{\p(\tau_x^+<\eq)}
&\geq& \frac{\p(\tau_x^+<\tau_{-\epsilon}^-\wedge\eq)}{\p(\tau_{-\epsilon}^-\geq \tau_x^+\wedge\eq)\p(\tau_x^+<\eq)}\q\mbox{and}\\
\frac{\p(\overline{D}_{\eq}>x)}{\p(\tau_x^+<\eq)}&\leq &\frac{\p(\tau_x^+<\tau_{-\epsilon}^-\wedge\eq)}{\p(\tau_{-\epsilon}^-\geq \tau_x^+\wedge\eq)\p(\tau_x^+<\eq)}
+\frac{
\p(\tau_x^+<\eq)- \p(\tau_{x+\epsilon}^+<\eq)}{\p(\tau_x^+<\eq)}.\label{lowerdwa}
\end{eqnarray}
The first terms on the right-hand sides of \eqref{upperdwa} and \eqref{lowerdwa}
may be simplified by using that, by the Markov property, we have
\begin{eqnarray}\label{simple}
\p(\tau_x^+<\tau_{-\epsilon}^-\wedge\eq)&=&\p(\tau_x^+<\eq)-\p(\tau_{-\epsilon}^-<\tau_x^+<\eq)\\
\nonumber &=&\p(\tau_x^+<\eq)-\e\left[\I_{\{\tau_{-\epsilon}^-<\tau_x^+\wedge\eq\}}\p_{\widehat{X}_{\tau_{-\epsilon}^-}}(\tau_x^+<\eq)\right].
\end{eqnarray}
Furthermore, since $\ovl\Pi\in\mathcal S^{(\alpha)}$ we note that
\begin{equation}\label{posr}
\lim_{x\to\infty}\frac{
\p(\tau_{x+\epsilon}^+<\eq)}{\p(\tau_x^+<\eq)}=\te{-\alpha \epsilon}, \q\epsilon>0.
\end{equation}
From the dominated convergence theorem
and Definition \ref{def3}(ii)--(iii) it then follows that
\begin{equation}\label{post2}
\lim_{x\to\infty}\frac{\e\left[\I_{\{\tau_{-\epsilon}^-<\tau_x^+\wedge\eq\}}\p_{\widehat{X}_{\tau_{-\epsilon}^-}}(\tau_x^+<\eq)\right]}{\p(\tau_x^+<\eq)}
=\e\left[\te{\alpha \widehat{X}_{\tau_{-\epsilon}^-}}\I_{\{\tau_{-\epsilon}^-<\eq\}}\right],
\end{equation}
and an application of the Markov property yields
\begin{equation}\label{Levyconference}
\e\left[\te{\alpha \widehat{X}_{\tau_{-\epsilon}^-}}\I_{\{\tau_{-\epsilon}^-<\eq\}}\right]=\frac{\e\left[\te{\alpha \underline{\widehat{X}}_{\eq}}\I_{\{\tau_{-\epsilon}^-<\eq\}}\right]}{
\e\left[\te{\alpha \underline{\widehat{X}}_{\eq}}\right]}.
\end{equation}
Taking first $x\to\infty$ in \eqref{upperdwa} and \eqref{lowerdwa} and using
\eqref{simple}, (\ref{posr}), \eqref{post2} and (\ref{Levyconference})
and that $\p[\tau_{-\epsilon}^-= \eq]=0$
we find
$$
\frac{\e\left[\left.\te{\alpha \underline{\widehat{X}}_{\eq}}\right| \tau_{-\epsilon}^- > \eq \right]}{\e\left[\te{\alpha \underline{\widehat{X}}_{\eq}}\right]}
\leq
\liminf_{x\to\infty} \frac{\p(\overline{D}_{\eq}>x)}{\p(\tau_x^+<\eq)}
\leq \limsup_{x\to\infty} \frac{\p(\overline{D}_{\eq}>x)}{\p(\tau_x^+<\eq)}
\leq \frac{\e\left[\left.\te{\alpha \underline{\widehat{X}}_{\eq}}\right| \tau_{-\epsilon}^- > \eq \right]}{\e\left[\te{\alpha \underline{\widehat{X}}_{\eq}}\right]} + 1 - \te{-\alpha\epsilon}.
$$
Letting subsequently
$\epsilon\downarrow 0$ and using
$$
\lim_{\epsilon\downarrow 0}\e\left[\left.\te{\alpha \underline{\widehat{X}}_{\eq}}\right| \tau_{-\epsilon}^- > \eq \right] =1,
$$
which in turn holds as the conditional expectation is bounded above by $1$ and bounded below by $\te{-\alpha\epsilon}$,
we get the following asymptotics:
\begin{eqnarray}\label{firstinc}
&&\p(\overline{D}_{\eq}>x)\simeq B_q \overline{\Pi}(x), \q\text{with}\\
&&B_q=\frac{\widehat{\kappa}(q,0)}{(\widehat{\kappa}(q,0)+\widehat{\kappa}(0,-\alpha))^2}
\frac{1}{\e\left[\te{\alpha \underline{\widehat{X}}_{\eq}}\right]}.\nonumber
\end{eqnarray}

\noindent
Next, we turn to the proof of the asymptotic decay of the second term on the right-hand side of \eqref{basicheavyover}.
Note that it equals
\begin{eqnarray}
\nn \lefteqn{\int_{[0,x]}\p(\overline{\widehat{X}}_\infty>x-z)\p(D_{\eq}\in\ud z,\overline{D}_{\eq}\leq x)}\\&&=\left(\int_{[0,y^\prime]}+\int_{(y^\prime,x-y^\prime]}+\int_{(x-y^\prime,x]}\right)\p(\overline{\widehat{X}}_\infty>x-z)\p(D_{\eq}\in\ud z,\overline{D}_{\eq}\leq x).\label{sumint}\end{eqnarray}
We next show that the second and third integral of the right-hand side of \eqref{sumint} tend to zero as we let first $x$ and then $y$ tend to infinity.
Indeed, concerning the second integral we use
 \eqref{Con1}, Definition \ref{def3}(ii)--(iii) and \eqref{jamaria2} to show
that
\[\lim_{x\to\infty}\frac{\int_{(y^\prime,x-y^\prime]}\p(\overline{\widehat{X}}_\infty>x-z)\p(D_{\eq}\in\ud z,\overline{D}_{\eq}\leq x)}{\p(\tau_x^+<\infty)}=\int_{(y^\prime,\infty)} \te{\alpha z}
\p(\overline{\widehat{X}}_{\eq} \in\ud z),\]
which tends to $0$ as $y^\prime \to\infty$.

For the third integral, we obtain the bound
\begin{eqnarray*}
\int_{(x-y^\prime,x]}\p(\overline{\widehat{X}}_\infty>x-z)\p(D_{\eq}\in\ud z,\overline{D}_{\eq}\leq x)&\leq&
\p(\overline{\widehat{X}}_\infty>y^\prime)\p(\overline{\widehat{X}}_{\eq} >x-y^\prime)\\ &\leq&
\p(\tau^+_{y^\prime} < \infty)\p(\tau^+_{x-y^\prime}<\infty).\end{eqnarray*}
After dividing the integral in the display by $\p(\tau_x^+<\infty)$ and letting first $x\to\infty$ and then
$y^\prime \to\infty$, it tends to zero.

Finally, the first integral on the right-hand side of \eqref{sumint}
is asymptotically of the same order as the left-hand side.
Indeed, using \eqref{Con1} and Definition \ref{def3}(ii)--(iii), \eqref{jamaria2} and
the dominated convergence theorem
we find
\begin{equation}\label{secinc}
\lim_{x\to\infty}\frac{\int_{[0,y^\prime]}\p(\overline{\widehat{X}}_\infty>x-z)\p(D_{\eq}\in\ud z,\overline{D}_{\eq}\leq x)}{\p(\tau_x^+<\infty)}
=\int_{[0,y^\prime]}\te{\alpha z}\p(D_{\eq}\in\ud z),\end{equation}
which converges to $\int_0^{\infty}\te{\alpha z}\p(D_{\eq}\in\ud z)=\e[ \te{\alpha \overline{\widehat{X}}{\eq}}]:=\WT B_q$
as $y^\prime \to\infty$.

By combining the previous estimates we have the following asymptotics of the tail probability $\p(-\underline{D}^*_{\eq}>x)$:
\begin{equation}\label{eq:conv}
\lim_{x\to\infty}\frac{\p(-\underline{D}^*_{\eq}>x)}{q\overline\Pi(x)} = q^{-1} (B_q + \WT B_q).
\end{equation}
Noting that the right-hand side of \eqref{eq:conv} is a pointwise limit of Laplace transforms of measures
and is itself such a Laplace transform, it follows from (an extension of) the continuity theorem
(see \cite[Theorem 15.5.2]{Kallenberg}) that
the corresponding measures also converge to the limiting measure with Laplace transform given
by $q^{-1} (B_q + \WT B_q)$.
Hence the first assertion of the theorem follows by inverting the Laplace transform $q^{-1} (B_q + \WT B_q)$
(see Remark \ref{remmq}).

Concerning $\overline D^*_t$, note that by \eqref{reprunderUstar} we have
\[\p(-\overline D^*_t>x)=\int_{(-\infty,0]}\p(\tau_{x+z}^+<\infty)\p(\underline{\widehat{X}}_t\in\ud z).\]
Asymptotics \eqref{jamaria2}, the dominated convergence theorem and part (ii) and (iii) of Definition \ref{def3}
establish that the asymptotic decay of $\p(-\overline D^*_t>x)$ is as stated.
\end{proof}

\section{Exact distributions}\label{Examplesexact}

From Proposition \ref{repr} it follows that the distributions of
$\overline{D}^*_{t,s}$, $\underline{D}^*_{t,s}$, $\overline{U}^*_{t,s}$ and $\underline{U}^*_{t,s}$
can be identified if one is able to identify the law of the
finite time supremum and the resolvent of the L\'{e}vy process reflected at its infimum.
In the case of a spectrally one-sided L\'evy process $X$ such explicit expressions are provided by
existing fluctuation theory.

In this section we suppose that $X$ is spectrally negative (as noted in the Introduction,
the case of spectrally positive L\'evy process follows from by considering the dual of $X$).
Many fluctuation results for $X$ can be conveniently formulated in terms of its scale function $W^{(q)}$ that is defined as the
unique continuous increasing function on $\mathbb R_+$ with Laplace transform
\[\int_0^\infty \te{-\lambda x}W^{(q)}(x)\,\ud x=\frac{1}{\psi(\lambda)-q}\quad\mbox{for any }\lambda> \Phi(q).\]
Note that by convexity of the Laplace exponent $\psi$ its right inverse
$\Phi(q)$ is well-defined for all $q\geq 0$. Moreover, let $Z^{(q)}$ denote the function on $\mathbb R_+$ given by
$$Z^{(q)}(x)= 1 + q\int_0^x W^{(q)}(y)\ud y, \quad x\in\mathbb R_+,$$
let $\eb$ be an exponentially distributed random variable with parameter $\beta>0$
(independent of $\eq$ and $X$).

\begin{prop}\label{prop:DUMfinite} Let $x\in\mathbb R_+$. {\bf (i)} If $\e[X_1] \in\mathbb R\cup\{-\infty\}\backslash\mathbb R_+$ then
\begin{eqnarray*}
\mathbb{P}(\overline U^*_{\eq,\eb }>
x)&=& \frac{1}{Z^{(q)}(x)}\left[1+q\int_0^x\te{-\Phi(\beta)z}W^{(\beta)}(z)\ud z\right]\q\mbox{and}\\
\mathbb{P}(\underline U^*_{\eq,\eb}> x)&=&\frac{q}{q-\beta}\te{-\Phi(\beta)x}\frac{\Phi(\beta)-\Phi(q)}{\Phi(q)}.
\end{eqnarray*}

{\bf (ii)} If $\e[X_1] \in\mathbb R_+\backslash\{0\}$ then
\begin{eqnarray*}
\mathbb{P}(- \overline
D^*_{\eq, \eb}>x)&=&\Phi(q)\int_0^\infty \te{-\Phi(q)z}Z^{(\beta)}(x+z)\ud z-\frac{\beta}{\Phi(\beta)}\Phi(q)\int_0^\infty
\te{-\Phi(q)z}W^{(\beta)}(x+z)\ud z\q\mbox{and}\\ \mathbb{P}(-\underline D^*_{\eq,\eb}>
x)&=&q\frac{\beta}{\Phi(\beta)}\int_{[0,x]}(W^{(\beta)}(x-z)-\beta Z^{(\beta)}(x-z)) W^{(q)}(z)\ud z
\\&&- \frac{\beta}{\Phi(\beta)}
\frac{W^{(q)}(x)}{W^{(q)\prime}_+(x)}\int_{[0,x]} (W^{(\beta)}(x-z)-\beta Z^{(\beta)}(x-z))W^{(q)}(\ud
z)\\&&\q+Z^{(q)}(x)-q\frac{W^{(q)}(x)^2}{W^{(q)\prime}_+(x)},
\end{eqnarray*}
where $W^{(q)\prime}_+(x)$ denotes the right-derivative of $W^{(q)}$ at $x$.
\end{prop}

The proof of Proposition \ref{prop:DUMfinite} is based on the representations
derived in Proposition \ref{repr} and the form of the $q$-resolvent measures
$R^U_x$ and $R^D_x$ of $U$ and $D$ killed upon crossing
the level $x>0$,  which are defined by
$$
R^U_x(\ud y)=\int_0^\infty \te{-qt}\p(U_t\in \ud y, T^U_x>t)\ud t
\qquad\mbox{and}\q
R^D_x(\ud y)=\int_0^\infty \te{-qt}\p(D_t\in \ud y, T^D_x>t)\ud t,
$$
where $T^U_x$ and $T^D_x$ are the first-passage times of $U$ and $D$ over $x$,
$T^U_x=\inf\{t\geq 0: U_t>x\},$ $T^D_x=\inf\{t\geq 0: D_t>x\}.$
In \cite[Theorem 1]{Pistorius} it was shown that these resolvent measures have a density a version of which is given by
\begin{eqnarray}
\label{RU}
R^U_x(\ud y) &=& \frac{W^{(q)}(x-y)}{Z^{(q)}(x)}\ud y, \qquad y\in[0,x],\\
\label{RD}
R^D_x(\ud y) &=& W^{(q)}(x)\frac{W^{(q)}(\ud y)}{W^{(q)\prime}_+(x)}
-W^{(q)}(y)\ud y, \qquad   y \in [0, x].
\end{eqnarray}

\begin{proof}
Recall that by Proposition \ref{repr} we have
\begin{equation*}
\p(\overline U^*_{\eq, \eb}>x)=\e\left[ \te{-q T^U_x}\right]
+q\int_{[0,x]}\p(\overline{X}_{\eb}>x-z)R^U_x(\ud z),
\end{equation*}
where by \cite[Proposition 2]{Pistorius},
\[ \e\left[ \te{-qT^U_x}\right]=\frac{1}{Z^{(q)}(x)}\]
and
\[\p(U_{\eb}>x-z)=\p(\overline{X}_{\eb}>x-z)=\te{-\Phi(\beta)(x-z)}.\]
Similarly,
\begin{equation*}
\p(\underline U^*_{\eq, \eb}>x)=\int_0^{\infty} \p(U_{\eb}>x+z)\p(D_{\eq}\in \ud z),
\end{equation*}
where by \cite{kyprpalm}:
\begin{equation*}
\p(D_{\eq}\in \ud z)=\p(-\underline{X}_{\eq}\in \ud z)=\frac{q}{\Phi(q)}\,W^{(q)}(\ud z)-q W^{(q)}(z)\,\ud z,
\quad z\in\mathbb R_+.
\end{equation*}
Straightforward calculations complete the proof of (i).

The proof of (ii) follows by a similar reasoning using the identity
\[\e\left[\te{-qT^D_x}\right]=   Z^{(q)}(x)-q\frac{W^{(q)}(x)^2}{W^{(q)\prime}_+(x)};\]
see \cite[Proposition 2]{Pistorius}.
\end{proof}

\begin{cor}\label{prop:DUM}Let $x\in\mathbb R_+$. {\bf (i)} If $\e[X_1] \in \mathbb R\cup\{-\infty\}\backslash\mathbb R_+$
\begin{eqnarray*}
\mathbb{P}(\overline U^*_{\eq}> x)&=&
\frac{1}{Z^{(q)}(x)}\left[1+q\int_0^x\te{-\Phi(0)z}W^{(q)}(z)\ud z\right]\q\mbox{and}
\\
\mathbb{P}(\underline U^*_{\eq}>
x)&=& \te{-\Phi(0)x}\frac{\Phi(q)-\Phi(0)}{\Phi(q)}.
\end{eqnarray*}

{\bf (ii)} If $\e[X_1] \in \mathbb R_+\backslash\{0\}$ then
\begin{eqnarray*}
\mathbb{P}(- \overline
D^*_{\eq}>x)&=&1-\psi'(0)\Phi(q)\int_0^\infty
\te{-\Phi(q)z}W(x+z)\ud z\q\mbox{and}\\
\mathbb{P}(-\underline D^*_{\eq}>
x)&=&1 + q\psi'(0)\int_{[0,x]}W(x-z) W^{(q)}(z)\ud z\\&& - \psi'(0)
\frac{W^{(q)}(x)}{W^{(q)\prime}_+(x)}\int_{[0,x]} W(x-z)W^{(q)}(\ud
z).
\end{eqnarray*}
\end{cor}

\begin{proof}
Note that by negative drift condition $\e[X_1] \in \mathbb R\cup\{-\infty\}\backslash\mathbb R_+$
we have that $\psi'(0)=\e[X_1]<0$ and by convexity of $\psi$ we can conclude that
$\Phi(0)>0$. Moreover, since $U_\infty$ has the same law as $\overline{X}_\infty$, which follows an exponential distribution with parameter $\Phi(0)$, we have
for any $x\in\mathbb R_+$
\begin{eqnarray*}
\mathbb{P}(\overline U^*_{\eq}\leq
x)&=&\int_{[0,x]}\int_0^y\Phi(0)\te{-\Phi(0)z}\ud
z\mathbb{P}(z+U_{\eq}\in \ud y, \eq<T_x^U)
\\
&=&
\frac{q\Phi(0)}{Z^{(q)}(x)}\int_0^x\int_0^y \te{-\Phi(0)z}W^{(q)}(x-y+z)\ud z\,\ud y\\
&=& \frac{1}{Z^{(q)}(x)}\le[q\int_0^x(1 - \te{-\Phi(0)y})W^{(q)}(y)\,\ud y\ri].
\end{eqnarray*}
Furthermore, from \eqref{specnegprefactors},
\begin{eqnarray*}
\mathbb{P}(\underline U^*_{\eq}\leq x)&=&\mathbb{P}(\widetilde U^*_0-D_{\eq}\leq x)\\
&=&
\Phi(0)\int_{\mathbb R_+}\int_{-y}^x \te{-\Phi(0)(z+y)}\ud z\p(D_{\eq}\in \ud y) = 1-\te{-\Phi(0)x}\frac{\Phi(q)-\Phi(0)}{\Phi(q)}.
 \end{eqnarray*}
The proof of (ii) follows by a similar reasoning, using the form  of the resolvent and
the fact that $D_\infty$ has the same law as $-\underline X_\infty$,
which is given by $\p[-\underline X_\infty<x] = \psi'(0)^{-1} W(x)$ for $x\in\mathbb R_+$
(see {\em e.g.} \cite{kyprpalm}), where we use fact that $\psi'(0)=\e[X_1]> 0$.
\end{proof}

\begin{rem}\rm
\begin{itemize}
\item[{\bf (i)}] By inverting the Laplace transform we find
that
\[\p(\underline U^*_t>x)=\te{-\Phi(0)x}\left(1-\Phi(0)\e
[U_t]\right).\]

\item[{\bf (ii)}] Straightforward calculations show that the double Laplace
transforms $\mathcal{L}_U(r,s)$ and $\mathcal{L}_D(r,s)$ of
$\p(\underline U^*_{T}\leq u)$ and $\p(-\underline D^*_{T}\leq u)$
in $T$ and $u$ are given by:
\begin{eqnarray*}
\mathcal{L}_U(r,s) = \frac{\Phi(0)(\Phi(s) + r)}{(\Phi(0) +
r)s\Phi(s)}, \qquad \mathcal{L}_D(r,s ) =  r\psi'(0){\Phi}(s
)\frac{\psi(r)-s }{s ^2\psi(r)(r-\Phi(s ))}.
\end{eqnarray*}
This agrees with the forms of $\mathcal{L}_D(r,s )$ and $\mathcal{L}_U(r,s)$ obtained in
\cite{Debicki}.
\item[{\bf (iii)}] In the literature numerical methods have been developed
for the evaluation of scale functions, based on Markov chain approximation (see \cite{MVJ})
or Laplace inversion (see \cite{KKR,S}), which may be used for numerical evaluation of
the expressions given in Proposition~\ref{prop:DUMfinite}.

\item[{\bf (iv)}] From the proofs of the propositions above it is clear that we can identify
the bivariate Laplace transform of $\overline U^*_{t,s}$, $\underline U^*_{t,s}$, $\overline D^*_{t,s}$
and $\underline D^*_{t,s}$
with respect of $t$ and $s$ as long as the laws of $\overline{X}_{\eq}$, $\underline{X}_{\eq}$ and resolvents of reflected process $R^U_a$, $R^D_a$
are known. This could be done not only for spectrally one-sided L\'evy processes. For example, one can consider the Kou model, where the log-price
$X=(X_t)_{t\in\mbb R_+}$ is modelled by a jump-diffusion
with constant drift $\mu$ and volatility $\sigma>0$, with
the upward and downward jumps arriving at rate
$\lambda_+$ and $\lambda_-$ with sizes following exponential
distributions with mean $1/\alpha_+$ and $1/\alpha_-$,
$$ X_t = \mu t + \sigma W_t + \sum_{j=1}^{N^+_t} U^+_j -
\sum_{j=1}^{N^-_t} U^-_j,$$
where $N^\pm$ are independent standard Poisson processes with rates
$\lambda^\pm$, independent of a Brownian motion $W$, and
$U^\pm_i\sim\mathrm{Exp}(\alpha^\pm)$ are independent.
Then the important ingredients are identified in \cite[Lemma 1 and Proposition 3]{sorenmartijn}
(also applied  for the dual process).
\end{itemize}
\end{rem}

\subsection{(Future) drawdowns and drawups under Black--Scholes model}\label{Examples}
Consider a risky asset whose price process $P=(P_t)_{t\in\mbb R_+}$ is given as follows:
\begin{equation}
P_t = P_0 \exp(X_t), \q t\in\mbb R_+,
\end{equation}
where $X=(X_t)_{t\in\mbb R_+}$ is a L\'{e}vy process.
In the case of the Black--Scholes model, $P$ is a geometric
Brownian motion, with rate of appreciation $\mu\in\mbb R$ and the volatility
$\sigma$, and $X=(X_t)_{t\in\mbb R_+}$ is given by the linear Brownian motion
$$ X_t = \le(\mu - \frac{\sigma^2}{2}\ri) t + \sigma W_t.$$
Let $\mu>\sigma^2/2$.
This model is widely used in practice as a benchmark
for other models.

For this model we have $\psi(\theta)=\sigma^2\theta^2/2+(\mu-\sigma^2/2)\theta$,
$\Phi(q)=-\omega+\delta(q)$
with \[\delta(q)=\sigma^{-2}\sqrt{(\mu-\sigma^2/2)^2+2\sigma^2 q}\] and $\omega =\frac{\mu}{\sigma^2} -\frac{1}{2}$ and
\[W^{(q)}(x)=\frac{1}{\delta(q)\sigma^2}\left[\te{(-\omega+\delta(q))x}-\te{-(\omega+\delta(q))x}\right],\]
\[Z^{(q)}(x)=\frac{q}{\delta(q)\sigma^2}\left[\frac{1}{-\omega+\delta(q)}\te{(-\omega+\delta(q))x}+\frac{1}{\omega+\delta(q)}\te{-(\omega+\delta(q))x}\right].\]

Hence from Corollary \ref{prop:DUM} we have
\begin{eqnarray*}
\lefteqn{\p(-\underline D^*_{\eq} > x) = 1+\frac{1}{\sigma^4\delta(q)\omega}(\mu-\sigma^2/2)(Z^{(q)}(x)-1)}\\
&&+\frac{q}{\sigma^4\delta(q)\omega}(\mu-\sigma^2/2)\left( \frac{1}{\delta(q)-\omega}\te{-(\delta(q)+\omega)x}-
\frac{1}{\delta(q)+\omega}\te{(\delta(q)-\omega)x}-\frac{2\omega}{\delta^2(q)-\omega^2}\te{-2\omega x}\right)\\&&
+(\mu-\sigma^2)\frac{W^{(q)}(x)}{W^{(q)\prime}(x)}\left(
\frac{\delta(q)+\omega}{\delta(q)-\omega}\te{-(\delta(q)+\omega)x}+\frac{\delta(q)-\omega}{\delta(q)+\omega}\te{(\delta(q)-\omega)x}
-\frac{2\delta(q)}{\delta^2(q)-\omega^2}\te{-2\omega x}\right)
\end{eqnarray*}
and
\begin{eqnarray*}
 \p(- \overline D^*_{\eq} > x) &=& \frac{-\omega+\delta(q)}{\omega+\delta(q)}\te{-2\omega x}, \q q >0.
\end{eqnarray*}
Hence we find for $t\in\mbb R_+$
\begin{eqnarray*}
\p(- \overline D^*_t > x) &=& \e[ \te{-2\omega U_t}]\te{-2\omega x}.
\end{eqnarray*}
Moreover,
\begin{eqnarray*}
\overline\e^{(\gamma)}[P_t] &=& P_0 \frac{\e\left[\te{-\gamma U_t + X_t}\right]}{\e\left[\te{-\gamma U_t}\right]} =
P_0 \te{\psi(1)t}\frac{\mathbb E^{(1)}[\te{-\gamma U_t}]}{\e\left[\te{-\gamma U_t}\right]},
\\ \underline\e^{(\gamma)}[P_t] &=& P_0\te{\psi(1)t} \frac{\mathbb E^{(1)}[\te{\gamma D_t}]}{\e\left[\te{\gamma D_t}\right]},
\end{eqnarray*}
where $\overline\e^{(\gamma)}$ and $\underline\e^{(\gamma)}$ are the expectations with respect of measures $\overline\p^{(\gamma)}$ and $\underline\p^{(\gamma)}$
given in (\ref{miara}) (for $\gamma$ given in Assumption \ref{A2}), respectively,
and the measure $\p^{(1)}$ is defined via $\p^{(1)}(A) = \e[\te{X_t-\psi(1)t}\mathbf 1_A]$ for $A\in\mathcal F_t$ and $\gamma=2\omega$.
Under $\p^{(1)}$ we have
\[X_t=\left(\mu-\frac{3}{2}\sigma^2\right)+ \sigma W_t.\]

We note that $\e[ e^{-2\omega U_t}]=\e[ e^{-\gamma U_t}]$ and
$\e^{(1)}[ e^{-\gamma U_t}]$ may be identified using
\cite[(1.1.3), p. 250]{Borodin}  and
$\e[e^{\gamma D_t}]$ and $\e^{(1)}[e^{\gamma D_t}]$
using \cite[(1.1.3), (1.2.3) p. 250-251]{Borodin}.

\section*{Acknowledgements}
This work is partially supported by the National Science Centre of Poland (NCN)under the grant DEC-2013/09/B/HS4/01496
(2014-2016). The second author author also kindly acknowledges partial support by the project RARE -318984, a Marie Curie IRSES Fellowship within the 7th European
Community Framework Programme. We are very grateful to the referee for insightful
comments on the original version of this document.\\


\begin{thebibliography}{1}
\bibitem{sorenmartijn}
{\sc Asmussen, S., Avram, F. and Pistorius, M.} (2004).
Russian and American put options under exponential phase-type Levy models.
{\it Stoch. Proc. Appl.} {\bf 109}, 79--111.

\bibitem{LP} {\sc Bertoin, J.} (1996). \textit{L\'{e}vy processes}. Cambridge University Press.

\bibitem{BertDon}
{\sc Bertoin, J. and Doney, R.} (1994). Cram\'er's estimate for
L\'evy processes. {\it Stat. Probab. Lett.} {\bf 21(5)}, 363--365.


\bibitem{Borodin}
{\sc Borodin, A.N. and Salminen, P.} (2002).
{\it Handbook of Brownian Motion - Facts and Formulae}. Second Edition, Birkh\"auser.

\bibitem{Had2}
{\sc Carr, P., Zhang, H. and Hadjiliadis, O.} (2011). Maximum drawdown insurance. {\it International Journal of Theoretical and Applied Finance} {\bf 14(8)},1195--1230.

\bibitem{savov}  {\sc Chan, T., Kyprianou, A.E. and Savov, M.} (2009).
Smoothness of scale functions for spectrally negative L\'evy processes.
{\it Probab. Th. Rel. Fields} {\bf 150 (3-4)}, 691--708.

\bibitem{CO} {\sc Cherny, V. and Ob\l oj, J.} (2013).
Portfolio optimisation under non-linear drawdown constraints in a semimartingale financial model.
{\it Fin. Stoch.} {\bf 17(4)}, 771--800.

\bibitem{CT}
{\sc Cont, R. and Tankov, P.} (2004).
\newblock {\em Financial Modeling with Jump Processes}.
\newblock Chapman \& Hall/CRC.

\bibitem{Debicki} {\sc D\c{e}bicki, K., Kosinski, K.M. and Mandjes, M.} (2011). On the infimum attained by a reflected L\'evy process. {\it Queueing Systems} {\bf 70(1)}, 23--25.


\bibitem{krzysiekbook}
{\sc D\c{e}bicki, K. and Mandjes, M.} (2015).
\newblock {\em Queues and L\'evy Fluctuation Theory}.
\newblock Springer.

\bibitem{eg}
{\sc Embrechts, P. and Goldie, C.M.} (1982). On convolution tails.
{\it Stoch. Proc. Appl.} {\bf 13}, 263--278.


\bibitem{TakisFossZachary}
{\sc Foss, S., Konstantopoulos, T. and Zachary, Z.} (2007).
Discrete and Continuous Time Modulated Random Walks with Heavy-Tailed Increments.
{\it J. Theor. Probab.} {\bf 20(3)}, 581--612.




\bibitem{Kallenberg}
{\sc Kallenberg, O.} (1976).
{\it Random measures.} Akademie-Verlag, Berlin.

\bibitem{KOP} {\sc Kardaras, C., Ob\l oj, J. and Platen, E.} (2014).
The num\'{e}raire property and long-term growth optimality for drawdown-constrained
investments. {\it Math. Finance} {\tt doi:10.1111/mafi.12081}

\bibitem{KyprKlup} {\sc Kl$\ddot{\textrm{u}}$ppelberg, C., Kyprianou, A.~E. and Maller, R.} (2004). Ruin probabilities
and overshoots for general L\'{e}vy insurance risk processes. {\it Ann. Appl. Probab.} {\bf 14(4)}, 1766--1801.

\bibitem{sergueidima}
{\sc Korshunov, D. and Foss, S.} (2011).
\textit{An Introduction to Heavy-Tailed and Subexponential Distributions}. Springer.

\bibitem{Mero}
{\sc Kuznetsov, A., Kyprianou, A.~E. and Pardo, J.~C..} (2012).
Meromorphic {L}{\'e}vy processes and their fluctuation identities.
{\it Ann. Appl. Probab.} {\bf 22(3)}, 881--904.

\bibitem{KKR}{\sc Kuznetsov, A., Kyprianou, A.~E. and Rivero, V.} (2013).
The Theory of Scale Functions for Spectrally Negative L\'{e}vy Processes.
{\it L\'{e}vy Matters II}, 97--186. Springer.

\bibitem{Kyprianou} {\sc Kyprianou, A.~E.} (2006). \textit{Introductory Lectures on Fluctuations of L\'evy Processes with Applications}. Springer.

\bibitem{kyprpalm}
{\sc Kyprianou, A.~E. and Palmowski, Z.} (2005).
A martingale review of some fluctuation theory
for spectrally negative L\'evy processes.
{\it S\'eminaire de Probabilit\'es}, {\bf XXXVIII}, 16--29.

\bibitem{LLZ}
{\sc Landriault, D., Li, B. and Zhang, H.} (2017). On magnitude, asymptotics and duration of drawdowns for
L\'{e}vy models. {\it Bernoulli} {\bf 23(1)}, 432--458.

\bibitem{MVJ} {\sc Mijatovi\'{c}, A., Vidmar, M. and Jacka, S.} (2015).
Markov chain approximations to scale functions of
L\'{e}vy processes. {\it
Stoch. Proc. Appl.} {\bf 125(10)}, 3932--3957.

\bibitem{MP} {\sc Mijatovi\'{c}, A. and Pistorius, M.R.} (2012).
On the drawdown of completely asymmetric L\'{e}vy processes.
{\it Stoch. Proc. Appl.} {\bf 122(11)}, 3812--3836.


\bibitem{jamaria}
{\sc Palmowski, Z. and Vlasiou, M.} (2011). A L\'{e}vy input model
with additional state-dependent services. {\it
Stoch. Proc. Appl.} {\bf 121(7)}, 1546--1564.

\bibitem{PP} {\sc Palmowski, Z. and Pistorius, M.R.} (2009).
Cram\'er asymptotics for finite time first passage probabilities of general L\'evy processes.
{\it Stat. Prob. Lett.} {\bf 79}, 1752--1758

\bibitem{Pistorius} {\sc Pistorius, M.R.} (2004). On exit and ergodicity of the spectrally negative L\'evy process reflected at its infimum. \textit{J. Theor. Probab.} \textbf{17}, 183--220.

\bibitem{PosVec} {\sc Pospisil, L., Vecer, J. and Hadjiliadis, O.} (2009).
Formulas for Stopped Diffusion Processes with Stopping Times based on Drawdowns
and Drawups. {\it Stoch. Proc. Appl.}, {\bf 119(8)}, 2563--2578.


\bibitem{reich}
{\sc Reich, E.} (1958). On the integrodifferential equation of
Tak\'{a}cs. {\it I. Ann. Math. Stat.} {\bf 29}, 563--570.

\bibitem{Sato}
{\sc Sato, K.} (1999).
\textit{L\'evy Processes and Infinitely Divisible Distributions.} Cambridge
University Press, Cambridge.

\bibitem{S} {\sc Surya, B.A.} (2008). Evaluating scale functions of
spectrally negative L\'{e}vy processes. {\it J. Appl. Prob.} {\bf 45},
135--149.

\bibitem{Vigon}
{\sc Vigon, V.} (2002).
Votre L\'{e}vy ramp-t-il?
{\it J. London Math. Soc.} {\bf 65}, 243--256.

\bibitem{Z}
{\sc Zhang, H.} (2015).
Occupation time, drawdowns and drawups for one-dimensional diffusion.
{\it Adv. Appl. Probab.} {\bf 47}, 210--230.


\bibitem{Had1}
{\sc Zhang, H., Leung, T. and Hadjiliadis, O.} (2013).
Stochastic modelling and fair valuation of drawdown insurance.
 {\it Insur. Math. Econ.} {\bf 53(3)}, 840--850.
\end{thebibliography}
\end{document}